\documentclass[11pt]{article}
\usepackage{graphicx} 
\usepackage{amsmath,amsthm}
\usepackage{amssymb}
\usepackage{todonotes}
\usepackage{bbm}
\usepackage{enumitem, hyperref}
\usepackage{dsfont}
\usepackage{tikz}
\usepackage{comment}
\usetikzlibrary{automata,positioning}

\newcommand*{\N}{\mathbb{N}}
\newcommand*{\Z}{\mathbb{Z}}

\newcommand*{\R}{\mathbb{R}}

\newcommand*{\e}{\mathrm{e}}

\newcommand*{\Lcal}{\mathcal{L}}
\newcommand*{\Zcal}{\mathcal{Z}}
\newcommand*{\E}{\mathbb{E}}
\newcommand*{\Prob}{\mathbb{P}}

\newcommand*{\GWT}{\mathbb{T}}

\definecolor{caribbeangreen}{rgb}{0.0, 0.8, 0.6}

\newtheorem{theorem}{Theorem}
\newtheorem{lemma}{Lemma}
\newtheorem{example}{Example}

\newtheorem{prop}{Proposition}

\usepackage[left=1in, right=1in, bottom = 1in, top = 1in]{geometry}
\title{Maximal and minimal displacement of supercritical branching random walks on free products of groups}
\author{Robin Kaiser, Martin Klötzer, Konrad Kolesko, Ecaterina Sava-Huss}
\date{\today}

\begin{document}
\maketitle

\begin{abstract}
We prove that the maximal and minimal displacement of branching random walks with mean offspring number $\rho>1$ on free products of finite groups grows linearly almost surely. More precisely, we establish that the linear speed for the maximal (respectively minimal) displacement is given by the largest (respectively smallest)  intersection point of the large deviation rate function of the underlying random walk with the horizontal line at height $\log\rho$. The proof is based on constructing an associated multitype branching process  which consists of particles that travel fast enough, and distinguishing the types via the suffix of the particles locations.
\end{abstract}

\section{Introduction}

Classical random walks in linear time can be generalized to random walks indexed by trees \cite{tree-index} by attaching i.i.d.random variables to each edge of the tree index set and evaluating the random walk by summing over paths in the tree index set. Random walks with Galton-Watson trees as their index sets are also called branching random walks, based on the random branching of particles in each time step. 

In this article we consider branching random walks on free products of finite groups, such that the underlying branching process is supercritical, that is, in expectation every particle has more than one offspring. Supercritical branching random walks exhibit a wide array of interesting phenomena, ranging from the limiting behaviour of associated martingales, such as the additive martingale \cite{addmart1,addmart2} or the derivative martingale \cite{derivmart}, over limit theorems for the behaviour of the traveling particles \cite{stam,limit1,limit2}, to convergence of the ensemble of particles to the boundary of the underlying graph \cite{boundaryconv}. This is not a complete list of questions, since a full account of the abundant research on supercritical branching random walks exceeds the scope of this introduction.
One aspect of particular interest analyzed in the current paper is the farthest distance particles in a supercritical branching random walk can travel, a quantity known as the \textit{maximal displacement}. This topic has been studied extensively on the real line \cite{maxdisreal1,maxdisreal2}, where, under suitable conditions, the maximal displacement is shown to grow linearly in time.  On $\R$, much more is known: the second-order asymptotics of the maximal displacement are logarithmic \cite{logcorr1,logcorr2}, and after rescaling, the trajectory of the fastest particle converges to a Brownian excursion \cite{maxbrown}. Furthermore, the branching random walk on $\R$ converges in law when viewed from its leftmost point \cite{maxconv}.

When looking at branching random walks on spaces other than $\mathbb{R}$, much less is known; we refer to \cite{BRW} and the references cited there for an overview of results on branching random walks. Recently, the maximal and minimal displacement of branching random walks on Galton-Watson trees has been investigated in \cite{maxdisgalton}. This work together with \cite{limit2} is the main motivation for the current article, where as underlying state space for the expanding cloud of particles, we consider \textit{free products of finite groups} (and their Cayley graphs). These structures possess a rich and intricate geometry at infinity, and random walks on them are well understood; see \cite{RWOnFreeProducts,LLNRWFreeProduct,RateOfEscapeFreeProduct,AssymptoticEntropyFreeProduct,FiniteRangeFreeProduct} for a selection of results in this direction. Concerning branching random walks on free products of finite groups, much less is known. In \cite{BRWOnFreeProducts}, the Hausdorff dimension of the limiting set of a branching random walk on free product of groups is investigated; the results have been generalized to hyperbolic groups in \cite{MR4630601}. Recently, in \cite{Multifractal} the authors extended the study of the properties of the limit sets, and investigated the multifractal spectrum of branching random walks on free groups.

To present the main results, we first briefly outline the framework and the standing assumptions on the branching random walks.
Let $r \in \mathbb{N}$ and let $G_1, \ldots, G_r$ be finite groups. Set $G = G_1 * \cdots * G_r$ to be their free product. We consider a probability measure $\mu$ on $G$ given as a convex combination of probability measures supported on the factors $G_1, \ldots, G_r$. That is, there exist probability measures $\mu_i$ on $G_i$ for each $i \in \{1, \ldots, r\}$ and coefficients $\alpha_1, \ldots, \alpha_r \in [0,1]$ with $\sum_{i=1}^r \alpha_i = 1$ such that
\begin{equation}\label{def: mu}
\mu=\sum_{i=1}^r \alpha_i \mu_i.
\end{equation}
Let $\pi$ be a probability distribution on $\mathbb{N}$ with mean $\mathbb{E}[\pi] = \rho$. We refer to $\pi$ as the offspring distribution, and denote by $\GWT$ the Galton–Watson tree generated from $\pi$. Let $(Y_n)_{n \in \mathbb{N}}$ be the random walk on $G$ with step distribution $\mu$, and let $(X_v)_{v\in\GWT}$ be the random walk indexed by $\GWT$ whose steps, conditionally on $\GWT$, are independent and distributed according to $\mu$. Thus, $(X_v)_{v \in \GWT}$ is the branching random walk (BRW) on $G$: the branching mechanism follows $\pi$, and particles move independently on $G$ with law $\mu$. The following standing assumptions on the branching process and the random walk will be in force for the rest of the paper.

\begin{enumerate}[label = (A\arabic*)]
\setlength\itemsep{0em} 
\item\label{A1} The Galton-Watson tree $\GWT$ is supercritical and has finite first moment: for $\rho = \sum_{k\in\N}k\pi(k)$, it holds $1 < \rho < \infty$.
\item Every particle has at least one offspring, that is
$\pi(0)=0$.\label{A2}
\item\label{A3} The step distribution $\mu$ of the random walk $(Y_n)_{n\in\N}$ defined in (\ref{def: mu}) is non-degenerate in the sense that every group is visited with positive probability, i.e.
$$\alpha_k = \mu(G_k) >0 \text{ for every } k\in \lbrace1,2,\dots,r\rbrace$$
and the probability measures $\mu_i$ are irreducible.

We also assume that, if $r=2$, at least one of the two groups has cardinality greater than $2$, otherwise the resulting walk might be a simple random walk on $\mathbb{Z}$, which is recurrent and has zero drift.
\end{enumerate}
Condition $\ref{A3}$ implies that the random walk $(Y_n)_{n\in\N}$ has positive drift, i.e.
$$\ell = \lim_{n\to\infty}\frac{\E[|Y_n|]}{n} > 0,$$
see for example \cite{RateOfEscapeFreeProduct}, where $|Y_n|$ is the (word) distance between $Y_n$ and the identity of $G$.

Our proofs rely on large deviation principles for the underlying random walk $(Y_n)_{n \in \mathbb{N}}$, which, in the context of random walks on free products of finite groups, were established in \cite{largedev}. With rate function $I$ describing the exponential decay of events of the form $\{|Y_n| > a n\}$ if $a\geq \ell$, we obtain estimates on the probability that particles in the branching process propagate rapidly.
Let us define
\begin{align*}
v_{\max}:=\sup\{x\in D_I:\,I(x)\leq \log\rho\},
\end{align*}
where $D_I$ is the domain of finiteness of $I$.
More precisely, $v_{\max}$ is defined as the largest intersection point between the graph of $I$ and the horizontal line $y = \log \rho$, whenever such a point exists. 
Our first main result asserts that, for supercritical branching random walks on free products of finite groups, the maximal displacement grows linearly, with growth rate given by $v_{\max}$.

\begin{theorem}\label{thm:main}
Assuming conditions \ref{A1}–\ref{A3}, the maximal displacement of the supercritical branching random walk $(X_v)_{v \in \GWT }$ on $G$ exhibits linear growth at speed $v_{\max}$, namely
$$
\lim_{n \to \infty} \frac{\max_{|v| = n} |X_v|}{n} = v_{\max} \quad \text{almost surely}.
$$
\end{theorem}
The proof of Theorem \ref{thm:main} is divided into two parts: the lower bound is Proposition \ref{prop:lower} and the upper bound is Proposition \ref{prop:upper}. The upper bound relies on estimating the probability that particles with speeds exceeding $v_{\max}$ exist, employing large deviation estimates. For the lower bound we borrow  ideas from \cite[Theorem 1.3]{BRW} and \cite{maxdisgalton}, though the intricate geometry of free products introduces additional challenges. We construct a multitype branching process that retains only particles moving faster than a fixed threshold speed. Types are assigned according to the factor group containing the particle’s suffix. For an appropriate choice of parameters, we show that this multitype process survives with positive probability. A $0$-$1$ law then upgrades this to an almost sure lower bound.

An analogous result for the minimal displacement is also established; we omit some details since the proof is similar to the $\mathbb{Z}$-case. The minimal displacement $v_{\min}$ is defined as the smallest intersection point of the graph of $I$ with the line $y = \log \rho$, i.e.
\begin{align}\label{eq:v-min}
v_{\min}:=\inf\{x\geq 0:\,I(x)\leq \log\rho\}.
\end{align}
Note that $I(0)=-\log\mathsf{r}$, where $\mathsf{r}$ is the spectral radius of the random walk $(Y_n)_n$. This implies that $v_{\min}=0$ if and only if $\rho\mathsf{r}> 1$. This regime for $\rho$ is referred to as the recurrent or strong survival case, and so it is consistent that here the minimal displacement has $0$ drift.
\begin{theorem}\label{thm:main-min}
Assuming conditions \ref{A1}–\ref{A3}, the minimal displacement of the supercritical branching random walk $(X_v)_{v \in \GWT }$ on $G$ exhibits linear growth at speed $v_{\min}$, namely
$$\lim_{n\rightarrow\infty} \frac{\min_{|v|=n}|X_v|}{n}=v_{\min},\quad \text{almost surely}.$$
\end{theorem}
    
\paragraph{Outline.} Section \ref{sec:prelim} introduces the concepts and tools used throughout the paper. In Section \ref{sec:large-dev}, we establish large deviation estimates for the hitting time of distance $n$, which are needed in constructing the associated multitype branching process. Section \ref{sec:main} contains the proof of Theorem \ref{thm:main}, organized into three parts: Subsection \ref{sec:0-1} proves a 0–1 law for the lower bound on the maximal displacement; Subsection \ref{sec:upper} establishes the upper bound, in Subsection \ref{sec:multi-typ} we construct the multitype branching process of fast-moving particles and we show that it survives with positive probability; Subsection \ref{sec:lower} derives the lower bound using the multitype process. Finally, in Section \ref{sec: min dis} we prove Theorem \ref{thm:main-min}.

\section{Preliminaries}\label{sec:prelim}

\textbf{Free products of groups.} One method of constructing new groups with tree-like properties from given ones is to take their free product.
Let $r\in\N$ and $G_1,...,G_r$ be finite groups. We assume that all the groups are disjoint; however, they are allowed to be isomorphic to each other. For all $k\in\{1,...,r\}$ we denote by $G_k^\times$ the set of all elements of $G_k$, except for the neutral element which is denoted by $e_k$, that is $G_k^\times:=G_k\backslash\{e_k\}$,
and we write $L:=\bigcup_{k=1}^r G_k^\times$.
The free product $G=G_1*...*G_r$ is defined as the set of finite words over the alphabet $L$, that is
$$G:=\bigl\{x_1x_2...x_n:\ n\in\N,\forall j\leq n: (x_j\in L \text{ and }\forall k\in\{1,...,r\}:x_j\in G_k^\times\Rightarrow x_{j+1}\notin G_k^\times)\bigr\}\cup\{e\},$$
where $e$ denotes the empty word.
In words, the free product consists of all words with letters in $L$, such that no two consecutive letters are from the same group. We define a group operation on $G$ via concatenation and subsequent reduction of words: given two words $x=x_1\ldots x_n$ and $y=y_1\ldots y_m$ in $G$, their concatenation is
$xy:=x_1\ldots x_n y_1\ldots y_m$.
The product of $x$ and $y$ is obtained by performing all possible reductions in the concatenation $xy$. We write $xy$ for the reduced word. 

\begin{example}
Consider the groups $G_1=\{e_1,a\}$ and $G_2=\{e_2,b,b^2\}$ where $e_1$ and $e_2$ is the identity of $G_1$ and $G_2$ respectively, and  $a^2=e_1$ and $b^3=e_2$. So $G_1\cong\Z\slash 2\Z$ and $G_2\cong \Z\slash 3\Z$.
For words $x=ab^2ab$ and $y=b^2a$, the concatenated word $xy$ is $xy=ab^2abb^2a$, so by performing all possible reductions $ab^2a(bb^2)a \rightarrow ab^2(aa)\rightarrow ab^2$, we obtain
$xy=ab^2$. 
\end{example}
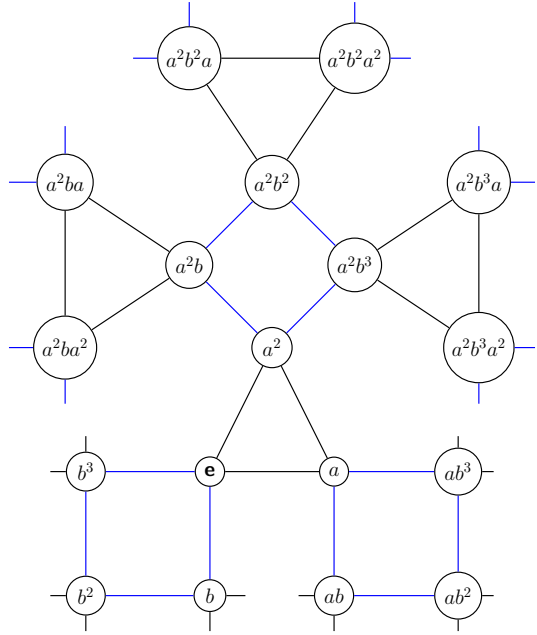
\begin{figure}[t]
\centering
\scalebox{.55}{\begin{tikzpicture}[->,>=stealth,auto,node distance=3cm,
  thick,main node/.style={circle,draw,font=\sffamily\Large\bfseries}]

  \node[main node] (1) at (0,0) {e};

  \node[main node] (2) at (3,0) {$a$};
  \node[main node] (3) at (1.5,3) {$a^2$};

  \node[main node] (4) at (0,-3) {$b$};
  \node[main node] (5) at (-3,-3) {$b^2$};
  \node[main node] (6) at (-3,0) {$b^3$};

  \node[main node] (7) at (3,-3) {$ab$};
  \node[main node] (8) at (6,-3) {$ab^2$};
  \node[main node] (9) at (6,0) {$ab^3$};

  \node[main node] (10) at (-0.5,5) {$a^2b$};
  \node[main node] (11) at (3.5,5) {$a^2b^3$};
  \node[main node] (12) at (1.5,7) {$a^2b^2$};

  \node[main node] (13) at (-3.5,7) {$a^2ba$};
  \node[main node] (14) at (-3.5,3) {$a^2ba^2$};

  \node[main node] (15) at (-0.5,10) {$a^2b^2a$};
  \node[main node] (16) at (3.5,10) {$a^2b^2a^2$};

  \node[main node] (17) at (6.5,7) {$a^2b^3a$};
  \node[main node] (18) at (6.5,3) {$a^2b^3a^2$};

  \node[] (19) at (0,-4) {};
  \node[] (20) at (1,-3) {};
  \node[] (21) at (-3,-4) {};
  \node[] (22) at (-4,-3) {};
  \node[] (23) at (-4,0) {};
  \node[] (24) at (-3,1) {};

  \node[] (25) at (3,-4) {};
  \node[] (26) at (2,-3) {};
  \node[] (27) at (6,-4) {};
  \node[] (28) at (7,-3) {};
  \node[] (29) at (7,0) {};
  \node[] (30) at (6,1) {};

  \node[] (31) at (-3.5,8.5) {};
  \node[] (32) at (-5,7) {};
  \node[] (33) at (-5,3) {};
  \node[] (34) at (-3.5,1.5) {};
  \node[] (35) at (-2,10) {};
  \node[] (36) at (-0.5,11.5) {};
  \node[] (37) at (5,10) {};
  \node[] (38) at (3.5,11.5) {};
  \node[] (39) at (8,7) {};
  \node[] (40) at (6.5,8.5) {};
  \node[] (41) at (8,3) {};
  \node[] (42) at (6.5,1.5) {};

  \path[-] (1) edge node {} (2);
  \path[-] (2) edge node {} (3);
  \path[-] (3) edge node {} (1);

  \path[-] (10) edge node {} (13);
  \path[-] (13) edge node {} (14);
  \path[-] (14) edge node {} (10);

  \path[-] (11) edge node {} (17);
  \path[-] (17) edge node {} (18);
  \path[-] (18) edge node {} (11);

  \path[-] (12) edge node {} (15);
  \path[-] (15) edge node {} (16);
  \path[-] (16) edge node {} (12);

  \path[-] (4) edge node {} (19);
  \path[-] (4) edge node {} (20);
  \path[-] (5) edge node {} (21);
  \path[-] (5) edge node {} (22);
  \path[-] (6) edge node {} (23);
  \path[-] (6) edge node {} (24);

  \path[-] (7) edge node {} (25);
  \path[-] (7) edge node {} (26);
  \path[-] (8) edge node {} (27);
  \path[-] (8) edge node {} (28);
  \path[-] (9) edge node {} (29);
  \path[-] (9) edge node {} (30);

  \path[-,blue] (1) edge node {} (4);
  \path[-,blue] (4) edge node {} (5);
  \path[-,blue] (5) edge node {} (6);
  \path[-,blue] (6) edge node {} (1);

  \path[-,blue] (2) edge node {} (7);
  \path[-,blue] (7) edge node {} (8);
  \path[-,blue] (8) edge node {} (9);
  \path[-,blue] (9) edge node {} (2);

  \path[-,blue] (3) edge node {} (10);
  \path[-,blue] (10) edge node {} (12);
  \path[-,blue] (12) edge node {} (11);
  \path[-,blue] (11) edge node {} (3);

  \path[-,blue] (13) edge node {} (31);
  \path[-,blue] (13) edge node {} (32);
  \path[-,blue] (14) edge node {} (33);
  \path[-,blue] (14) edge node {} (34);
  \path[-,blue] (15) edge node {} (35);
  \path[-,blue] (15) edge node {} (36);
  \path[-,blue] (16) edge node {} (37);
  \path[-,blue] (16) edge node {} (38);
  \path[-,blue] (17) edge node {} (39);
  \path[-,blue] (17) edge node {} (40);
  \path[-,blue] (18) edge node {} (41);
  \path[-,blue] (18) edge node {} (42);
\end{tikzpicture}}
\caption{A part of the Cayley graph of the free product $\Z/3\Z*\Z/4\Z$; $a$ is the generator of $\Z/3\Z$ and $b$ is the generator of $\Z/4\Z$, and $S=\{a,a^{-1},b,b^{-1}\}$. Black edges indicates multiplication by  $a$ or $a^{-1}$, and blue edges indicate multiplication by $b$ or $b^{-1}$.}\label{fig:free-prod}
\end{figure}
\textbf{Cayley graphs of free products.} For finite groups $G_1,...,G_r$ and $k\in\{1,...,r\}$, let $S_k\subseteq G_k^\times$ be a symmetric set that generates $G_k$, i.e. $\langle S_k \rangle = G_k$. For $G=G_1*...*G_r$, we take the set of generators $S=\bigcup_{k=1}^r S_k$.
The Cayley graph $\text{Cay}(G,S)$ is the graph with vertex set given by the group elements of $G$, and between vertices $x,y\in G$ there is an edge if $x^{-1}y\in S$; see Figure \ref{fig:free-prod}. For $x,y\in G$, denote by $path(x,y)$ the set of all paths in $\text{Cay}(G,S)$ from $x$ to $y$. The word length $|x|$ of $x\in G$ is
$|x|:=\inf_{\gamma\in path(e,x)}\text{len}(\gamma),$
where $\text{len}(\gamma)$ denotes the length (the number of edges) of the path $\gamma$. The distance between $x,y\in G$ is defined as
$d(x,y):=|x^{-1}y|$.
We denote by $\mathcal{B}_n:=\{x\in G:|x|<n\}$
the set of all words of length less than $n$ in $G$. For $x,y\in G$, we write $x\leq y$ if $x$ lies on a path from $e$ to $y$. The type $s(x)$ of a group element $x=x_1x_2\dots x_n\in G$ is the index of the group the last letter $x_n$ of $x$ belongs to, i.e.
$$s(x) = i \ :\iff x_n\in G_i.$$ 
We write $x_1$ for the first letter in the word $x$,
and we define the cone of type $i$ (for $1\leq i\leq r$) as
\begin{equation}\label{eq: cone def}
C(i):=\{y\in G : y_1\notin G_i\}.
\end{equation}

\textbf{Random walks on free products.}\label{par: RWonFreeProd} Let $\alpha_1, \ldots, \alpha_r \in (0,1)$ with $\sum_{i=1}^r \alpha_i = 1$. For each finite group $G_k$ with $k \in \{1, \ldots, r\}$, choose an irreducible probability measure $\mu_k \in \mathsf{Prob}(G_k)$, and set
$\mu := \sum_{i=1}^r \alpha_i \mu_i.$
The random walk $(Y_n)_{n \in \mathbb{N}}$ on $G$ with law $\mu$ is the Markov chain whose one-step transition probabilities are $p(x,y) := \mathbb{P}(Y_{n+1} = y \mid Y_n = x) = \mu(x^{-1} y)$ for $x, y \in G$.
In each step, we first choose a group $G_k$ with probability $\alpha_k$, and then sample an element from $G_k$ according to $\mu_k$. Denote the increments of the random walk by $(\xi_n)_{n \in \mathbb{N}}$, where $\xi_n = Y_{n-1}^{-1} Y_n$.
We refer to such random walks $(Y_n)$ as being of nearest-neighbor type on $\mathrm{Cay}(G,S)$. Since the factors $G_1, \ldots, G_r$ are finite, $(Y_n)_{n \in \mathbb{N}}$ has bounded step size: there exists $K \in \mathbb{R}$ with $\mathbb{P}(|\xi_1| \le K) = 1$. One may take
$$
K := \max\{|x| : \mu(x) > 0\} \le \max\{|x| : x \in L\} < \infty.
$$
The spectral radius $\mathsf{r}$ of $(Y_n)$ is defined by
$\mathsf{r} := \limsup_{n \to \infty} \mathbb{P}(Y_n = e)^{1/n}.$
Moreover, if $G$ has exponential growth (i.e., $|\mathcal{B}_n|$ grows exponentially in $n$) and $\operatorname{supp}(\mu)$ generates $G$, then $\mathsf{r} \in (0,1)$.

\begin{figure}
\centering
\includegraphics[scale=0.5]{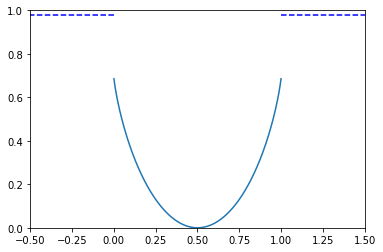}
\caption{The rate function on $\Z$ with increments ${\Prob(\xi_1=0)=\Prob(\xi_1=1)=\frac{1}{2}}$ is $I(x)=x\log(x)+(1-x)\log(1-x)+\log(2)$. The dashed lines indicate that $I$ equals $\infty$ at these values. The drift of the walk is $1/2$, and this is also a root of the rate function.}
\end{figure}

\textbf{Large deviations and rate functions.} For the random walk $(Y_n)_{n \in \mathbb{N}}$ on $G$, we define, for each $n \in \mathbb{N}$, the logarithmic moment generating function
\begin{align*}
	\Lambda_n:\R\rightarrow\R \quad \text{given by}\quad \Lambda_n(t)=\frac{1}{n}\log\bigl(\mathbb{E}\big[e^{t|Y_n|}\big]\bigr)
\end{align*}
and its limit
\begin{align}\label{eq:lim_log_moment}
\Lambda(t)=\lim_{n\rightarrow\infty}\Lambda_n(t),\end{align}
which exists for all $t\in \R$ due to the subadditivity and Fekete's lemma. In view of \cite{largedev}, if \ref{A3} holds, then the sequence $\big(\frac{|Y_n|}{n}\big)_{n\in\N}$  fulfills a large deviation principle with rate function $I$ given as the Legendre-Fenchel transform of the function $\Lambda$ defined in \eqref{eq:lim_log_moment}, that is, for all $x\geq 0$
\begin{align}
	\label{eq:Legendre_transform}
	I(x):=\sup_{t\in \R}\{xt-\Lambda(t)\}.	
\end{align}
In particular, for the rate function $I$ and every measurable set $A\subseteq [0,\infty)$, it holds
$$-\inf_{x\in A^\circ}I(x)\leq \liminf_{n\rightarrow\infty}\frac{1}{n}\log\mathbb{P}\bigg(\frac{|Y_n|}{n}\in A\bigg)\leq \limsup_{n\rightarrow\infty}\frac{1}{n}\log\mathbb{P}\bigg(\frac{|Y_n|}{n}\in A\bigg)\leq -\inf_{x\in\overline{A}}I(x),$$
where $A^\circ$ is the interior and $\overline{A}$ is the closure of $A$.
The function $I$ describes the decay rate of the probability of rare events of the form $\{ |Y_n|\geq an \}$ for $a\geq \ell$, and the decay of the probability of rare events of the form $\{|Y_n| \leq an\}$ for $0 < a\leq \ell$, as $n\to\infty$.
We collect below some basic facts about the rate function $I$.
\begin{lemma}\label{lem:rate-beh}
Let $(Y_n)_{n \in \mathbb{N}}$ be a random walk with step distribution $\mu$ satisfying \ref{A3} on the free product $G$. Let $\ell > 0$ denote its drift, i.e., $\lim_{n \to \infty} \frac{|Y_n|}{n} = \ell$ almost surely. Moreover, let $I$ be the rate function of $\big(\frac{|Y_n|}{n}\big)_{n \in \mathbb{N}}$. 
\begin{enumerate}[label=(I\arabic*)]
\setlength\itemsep{0em}
\item Let $D_I = \{x \in [0,\infty) : I(x) < \infty\}$ be the effective domain of $I$. Then $D_I$ is an interval with lower endpoint $0$ and upper endpoint $\beta \in (0,\infty)$, and $I$ is continuous on $D_I$.\label{I1}
\item $I(\ell)=0$.\label{I2}
\item $I$ is strictly increasing on $[\ell,\beta)$.\label{I3}
\item The function $x\mapsto \frac{I(x)}{x}$ is strictly increasing on $[\ell,\beta)$.\label{I4}
\end{enumerate}
\end{lemma}
\begin{proof}
\textit{\ref{I1}:} It follows from \cite[Proposition 5.1.]{largedev} that $I$ is convex, and thus  $D_I$ is a convex subset of $[0,\infty)$. Since the random walk $(Y_n)$ has finite jump size almost surely, $D_I$ is bounded. Indeed, for
$K=\sup\{|g|:g\in\text{supp}(\mu)\}$, from the triangular inequality we get $|Y_n|\leq Kn$ almost surely, which implies that $I=\infty$ on $(K,\infty)$. Thus, $D_I$ is a bounded interval. It also holds that $0\in D_I$, since $I(0)=-\log \mathsf{r}$, where $\mathsf{r}\in(0,1)$ is the spectral radius of $Y_n$. Finally, convexity of $I$ on $D_I$ implies continuity of $I$ on $D_I$. 

\textit{\ref{I2}:} For all open subsets $V\subseteq \R$ with $\ell\in V$ it holds 
$$\liminf_{n\rightarrow\infty}\frac{1}{n}\log\mathbb{P}\Big(\frac{|Y_n|}{n}\in V\Big)=0,$$
which together with \cite[Theorem 4.1.11.]{largedev_dembo} implies that the rate function at $\ell$ is given by
$$I(\ell)=\sup_{V\text{ open, }x\in V}-\liminf_{n\rightarrow\infty}\frac{1}{n}\log\mathbb{P}\Big(\frac{|Y_n|}{n}\in V\Big)=0.$$

\textit{\ref{I3}:} It suffices to show that for any $x>\ell$ it holds $I(x)>0$. Then the strict monotonicity follows from the convexity of $I$. First, assume that $\beta > \ell$, otherwise it is clear that $I$ is strictly increasing on $[\ell,\beta)$. For any $x\in (\ell,\beta)$, in view of the continuity of $I$ on $(0,\beta)$ we have
$$-I(x)=\lim_{n\to\infty} \frac{1}{n}\log\Prob(|Y_n|\geq nx).$$
For $i,m\in\N$ we set $Y_i^{(m)}:=\xi_{im}\xi_{im+1}\cdots\xi_{(i+1)m-1}$. Then  the  random variables  $(Y_i^{(m)})_{i\in \mathbb{N}}$ are independent and distributed as $Y_m$. Let $I_m$ be the large deviations rate function of a random walk on $\Z$ with increments distributed as $|Y_m|$. Then  $I_m(y)>0$ if $y>\E\big[|Y_m|\big]$. Moreover, from Fekete's lemma we have $\lim_{m\to\infty}\frac{\E|Y_m|}{m} = \ell$.
Let $x>\ell$ and choose $m\in\N$ such that $mx>\E\big[|Y_m|\big]$. Then 
\begin{align*}
   - I(x)&=\lim_{n\to\infty} \frac{1}{n}\log\Prob(|Y_n|\geq nx)=\lim_{n\to\infty} \frac{1}{nm}\log\Prob(|Y_{nm}|\geq nmx)\\
    &\leq\frac{1}{m}\lim_{n\to\infty}\frac{1}{n}\log\Prob\Big(\sum_{i=1}^n|Y_i^{(m)}|\geq nmx\Big)=-\frac{I_m(mx)}{m}<0,
\end{align*}
which implies that $I(x)>0$ and this finishes the proof.

\textit{\ref{I4}:} Let $\ell\leq x<y<\beta$.
The convexity of $I$ implies that $x\mapsto \frac{I(x)}{x-\ell}=\frac{I(x)-I(\ell)}{x-\ell}$ is nondecreasing. Since $\ell>0$ and $I(y)>0$ by \textit{\ref{I3}}, we have
$$\frac{I(x)}{x}=\frac{I(x)}{x-\ell}\frac{x-\ell}{x}\leq\frac{I(y)}{y-\ell}\frac{x-\ell}{x}<\frac{I(y)}{y-\ell}\frac{y-\ell}{y}=\frac{I(y)}{y}.$$
\end{proof}

\textbf{Branching random walks (shortly BRW).} 
Let $\pi$ be a probability distribution on $\mathbb{N} \cup \{0\}$ with mean $\rho = \mathbb{E}[\pi]$. Assume $\rho > 1$ and $\pi(0) = 0$. We define the Galton–Watson tree $\mathbb{T}$ -- known in the literature as the Harris--Ulam tree -- with offspring law $\pi$ inductively as follows. The vertex set is $\mathbb{N}^* = \bigcup_{k=1}^\infty \mathbb{N}^k \cup \{\emptyset\}$, with root $\emptyset$. A vertex $v$ at level $n$ independently produces a random number of children $\xi \sim \pi$; its offspring are added as concatenations $vk$ for $k \le \xi$. This means that each individual lives for exactly one time unit, produces $\xi$ offspring, and immediately dies afterwards.  For $v \in \mathbb{T}$, its level is $|v|$ (the word length), and the set of vertices at level $n$ is
$$\mathbb{T}_n := \{v \in \mathbb{T} : |v| = n\}.$$
For $v \in \mathbb{T}$ and $k \le |v|$, let $v_k$ denote the prefix of $v$ of length $k$, i.e., the ancestor of $v$ in generation $k$. Given the Galton–Watson tree $\mathbb{T}$, we define a random walk indexed by $\mathbb{T}$, also known as a branching random walk. Let $(\zeta_v)_{v \in \mathbb{T}}$ be i.i.d. with $\zeta_v \sim \mu$ for every $v \in \mathbb{T}$. The branching random walk indexed by $\mathbb{T}$ is the collection $(X_v)_{v \in \mathbb{T}}$ defined by
$$X_v = \zeta_{v_1}\,\zeta_{v_2}\cdots \zeta_{v} \quad \text{for } v \in \mathbb{T}, \qquad X_{\emptyset} = e.$$
If the branching random walk on $\mathbb{T}$ starts from $x \in G$ (i.e., $X_{\emptyset} = x$), we denote it by $(X_v^x)_{v \in \mathbb{T}}$. Under assumptions \ref{A1} and \ref{A2}, the Galton–Watson tree $\mathbb{T}$ survives almost surely. The process $(X_v^x)_{v \in \mathbb{T}}$ is the branching random walk on $G$ with offspring law $\pi$ and step law $\mu$; we write $\mathrm{BRW}(G,\pi,\mu)$ for short. A central tool in our analysis is the many-to-one formula, which reduces the first moment study of the supercritical branching random walk $(X_v)_{v \in \mathbb{T}}$ to that of the single random walk $(Y_n)$. While the formula holds more generally (see \cite{manytofew}), we state only the special case needed here.
\begin{lemma}\label{lem: many-to-one}
Under assumption \ref{A2}, for any measurable function $f:G\to \R$ it holds
$$\E\Big[\sum_{|v|=n}f(X_v)\Big]=\rho^n\E[f(Y_n)].$$
\end{lemma}
\begin{proof}
Conditioning on the $\sigma$-algebra  $\mathcal{F}_n$ generated by all branching particles up to generation $n$
$$\E\Big[\sum_{|v|=n}f(X_v)\,\big|\,\mathcal{F}_n\Big] = \sum_{|v|=n}\E[f(X_v)|\mathcal{F}_n] = |\GWT_n|\cdot\E[f(Y_n)].$$
Taking the expectation yields the claim.
\end{proof}

\section{Large deviation estimates}\label{sec:large-dev}

By \cite{largedev}, the sequence $\big(\frac{|Y_n|}{n}\big)_{n \in \mathbb{N}}$ satisfies a large deviation principle with rate function $I$, whose effective domain $D_I$ is a bounded interval (see Lemma \ref{lem:rate-beh}). In this section, we apply large deviation techniques to estimate the first exit time $T_n$ from the ball of radius $n$ in $G$, which is defined as
$$
T_n := \inf\{k \in \mathbb{N} : |Y_k| \ge n\}.
$$
We begin by studying the asymptotic behavior of the probability that $T_n$ is unusually small. Heuristically, a rapid exit from the ball of radius $n$ means the word length grows abnormally fast; the probability of this event can be quantified via the rate function $I$.
\begin{lemma}\label{lemm:ldp-tn}
Assuming \ref{A3}, for every $a \in [\ell, \beta)$, we have 
$$\lim_{n\to\infty}\frac{1}{n}\log\Prob\Big(T_n\leq \frac{n}{a}\Big) = -\frac{I(a)}{a}.$$
\end{lemma}
\begin{proof}
The lower bound is a direct consequence of the fact that for a fixed $a \in [\ell, \beta)$ the event $T_n \le n/a$ is implied by $|Y_{\lfloor n/a \rfloor}| \ge n$. For the upper bound note that
$$
\Big\{T_n\leq\frac{n}{a}\Big\} = \bigcup_{k=0}^{\lceil n/a\rceil}\{T_n = k\}\subset \bigcup_{k=0}^{\lceil n/a\rceil}\{|Y_k|\geq n\},
$$
since $T_n = k$ implies $|Y_k|\geq n$. Recall the definition of $\Lambda$ from \eqref{eq:lim_log_moment}. This definition implies that for every $t>0$, there exists a sequence $(\psi(k,t))_k$ of at most subexponential growth such that for all $k\in\N$ it holds
$$\E\big[\e^{t|Y_k|}\big]\leq \psi(k,t)\e^{\Lambda(t)k}.$$
Since $\Lambda$ is convex, for any $a\in (\ell,\beta)$ the supremum in \eqref{eq:Legendre_transform} (with $x$ replaced by $a$) must be achieved at some  $t_*>0$. Setting $\psi(k):=\psi(k,t_*)$, a union bound together with Markov inequality yields
\begin{align*}
\Prob\Big(T_n\leq \frac{n}{a}\Big)&\leq \sum_{k=0}^{\lceil n/a \rceil} \Prob(|Y_k|\geq n)\leq \sum_{k=0}^{\lceil n/a \rceil} \frac{\psi(k)\e^{\Lambda(t_*)k}}{\e^{t_*n}} \leq \frac{\e^{\Lambda(t_*)(n/a+1)}}{\e^{t_* n}}\sum_{k=0}^{\lceil n/a\rceil}\psi(k)\\
&= \e^{-n(t_* a-\Lambda(t_*))/a}\sum_{k=0}^{\lceil n/a\rceil}\e^{\Lambda(t_*)}\psi(k) = \e^{-nI(a)/a}\sum_{k=0}^{\lceil n/a\rceil}\e^{\Lambda(t_*)}\psi(k).
\end{align*}
Taking logarithm and dividing by $n$ yields the desired upper bound, since the partial sum of a subexponentially growing sequence grows at most subexponentially.
\end{proof}
For constructing the multitype branching process in Section \ref{sec:multi-typ}, we also need an estimate on the probability that the random walk $(Y_n)$ exits the ball of radius $n$ without ever leaving the cone $C(i)$, for each $i \in \{1,2,\dots,r\}$. Define the event $E_{n,i}$ of staying in $C(i)$ up to the first exit from $\mathcal{B}_n$ by
$$
E_{n,i} := \{\forall\, k \in \{1,2,\dots,T_n\}: \ Y_k \in C(i)\}.
$$
Recall the definition of $C(i)$ from \eqref{eq: cone def}, as well as the definition of $G_i$ as the $i$-th factor in the free product $G = G_1*G_2*\cdots*G_r$.
\begin{lemma}\label{lem: help large dev}
Let $a \in [\ell, \beta)$ and $i \in \{1,2,\dots,r\}$. Under assumption \ref{A3}, we have
$$
\lim_{n \to \infty} \frac{1}{n} \log \mathbb{P}\Big(T_n \le \frac{n}{a},\, E_{n,i}\Big) = -\frac{I(a)}{a}.
$$
\end{lemma}
\begin{proof}
Clearly, it suffices to prove that 
$$
\liminf_{n\to\infty}\frac{1}{n}\log\Prob\Big(T_n\leq \frac{n}{a},E_{n,i}\Big)\geq -\frac{I(a)}{a},
$$
because omitting $E_{n,i}$ only increases the probability, and then together with Lemma \ref{lemm:ldp-tn} we obtain the claim.
Fix $i \in \{1,2,\dots,r\}$ and let $d = \max\{|x| : x \in G_i\}$ denote the diameter of $G_i$. We first show 
\begin{equation}\label{eq:1inequality-lemma}
\Prob\Big(T_n\leq\frac{n-d}{a}\Big)\leq\#G_i\cdot\Big(\frac{n}{a}+2\Big)\cdot\Prob\Big(T_{n-d}\leq\frac{n-d}{a},\ \forall k\in \{1,2,\dots,T_{n-d}\}:Y_k\in C(i)\Big).
\end{equation}
Let $\tau$ be the last time the walk $(Y_n)$ visits $G_i$ before its first exit from $\mathcal{B}_n$. Here $G_i$ is viewed as a finite subgroup of $G$: it consists of all words of length 1 whose sole letter lies in $G_i$ (i.e., the copy of $G_i$ attached to the origin in the Cayley graph of $G$). Then we have
$$
\Big\{T_n \le \frac{n - d}{a}\Big\}
=
\bigcup_{k=0}^{\lfloor n/a \rfloor} \ \bigcup_{x \in G_i}
\Big\{T_n \le \frac{n - d}{a},\ \tau = k,\ Y_\tau = x\Big\}.
$$ 
After leaving $G_i$ for the last time (at time $k$) before leaving $\mathcal{B}_n$ the process $(Y_k^{-1}Y_{j+k})_{0\leq j \leq T_n}$ is trapped inside the cone $C(i)$, and so using the Markov property we have
\begin{align*}
\Prob\Big(T_n\leq\frac{n-d}{a},\ \tau = k,\ Y_\tau = x\Big) &\leq \Prob\Big(T_{n-|x|}\leq \frac{n-d}{a}-k,\ \forall l\in\{1,2,\dots,T_{n-|x|}\}:Y_l\in C(i)\Big)\\
& \leq \Prob\Big(T_{n-d}\leq \frac{n-d}{a},\ \forall k\in\{1,2,\dots,T_{n-d}\}:Y_k\in C(i)\Big).
\end{align*}
In the first inequality above, we use the fact that if $(Y_n)_n$ visits $G_i$ at time $k$, on site $x$, for the last time before leaving $\mathcal{B}_n$, then $T_n$ is equal to $k$ plus the first time at which the restarted process $(Y_k^{-1}Y_{k+j})_{j\geq 0}$ exits $\mathcal{B}_n$. Applying the union bound to both sums above completes the proof of \eqref{eq:1inequality-lemma}. To finish the proof of the lemma, it remains to show that
$$
\lim_{n \to \infty} \frac{1}{n} \log \mathbb{P}\Big(T_n \le \frac{n - d}{a}\Big) = -\frac{I(a)}{a}.
$$
For any $\delta > 0$ and all sufficiently large $n$, we have $\frac{n}{a+\delta} \le \frac{n - d}{a}$. Then
$$-\frac{I(a+\delta)}{a+\delta}\leq\liminf_{n\to\infty}\frac{1}{n}\log\Prob\Big(T_n\leq\frac{n-d}{a}\Big).$$
Continuity of $I$ on $D_I$ together with the  matching upper bound finishes the proof.
\end{proof}
Next we bound the probability of the number of branching particles that move quickly, but not excessively so.
\begin{lemma}\label{lem: LDP stoping time}
Let $a \in [\ell, \beta)$, $i \in \{1,2,\dots,r\}$, and choose $\varepsilon > 0$ with $a + \varepsilon < \beta$. Under assumption \ref{A3}, we have
$$
\lim_{n \to \infty} \frac{1}{n} \log \mathbb{P}\Big(\frac{n}{a+\varepsilon} < T_n \le \frac{n}{a},\ E_{n,i}\Big) = -\frac{I(a)}{a}.
$$
\end{lemma}
\begin{proof}
Let $a\in[\ell,\beta)$ and $i\in\{1,2,\dots,r\}$.
Lemma \ref{lem: help large dev} together with Lemma \ref{lem:rate-beh}\ref{I4} yield
$$
\lim_{n\to\infty}\frac{1}{n}\log\frac{\Prob\Big( T_n\leq\frac{n}{a+\varepsilon},\ E_{n,i}\Big)}{\Prob\Big( T_n\leq\frac{n}{a},\ E_{n,i}\Big)} = \frac{I(a)}{a}-\frac{I(a+\varepsilon)}{a+\varepsilon}<0,
$$
which implies that
\begin{equation}\label{eq: Lemma 4, 1}
\lim_{n\to\infty}\frac{\Prob\Big( T_n\leq\frac{n}{a+\varepsilon},\ E_{n,i}\Big)}{\Prob\Big( T_n\leq\frac{n}{a},\ E_{n,i}\Big)} = 0.
\end{equation}
Thus
\begin{align*}
\lim_{n\to\infty}\frac{1}{n}\log\Prob\Big(\frac{n}{a+\varepsilon}< & T_n\leq\frac{n}{a},\ E_{n,i}\Big)
=\lim_{n\to\infty}\frac{1}{n}\log\left(\Prob\Big( T_n\leq\frac{n}{a},\ E_{n,i}\Big)-\Prob\Big( T_n\leq\frac{n}{a+\varepsilon},\ E_{n,i} \Big)\right)\\
&=\lim_{n\to\infty}\frac{1}{n}\log\Prob\Big( T_n\leq\frac{n}{a},\ E_{n,i}\Big)+\lim_{n\to\infty}\frac{1}{n}\log\Bigg(1-\frac{\Prob\Big( T_n\leq\frac{n}{a+\varepsilon},\ E_{n,i}\Big)}{\Prob\Big( T_n\leq\frac{n}{a},\ E_{n,i}\Big)}\Bigg)\\
&=\lim_{n\to\infty}\frac{1}{n}\log\Prob\Big(T_n\leq\frac{n}{a},\ E_{n,i}\Big) = -\frac{I(a)}{a}.
\end{align*}
\end{proof}

\section{Maximal displacement for BRW}\label{sec:main}

We now begin the proof of Theorem \ref{thm:main}. As noted, the upper bound follows standard methods, whereas the lower bound requires a different approach. We first prove the independence of the starting position for the maximal displacement  of the branching random walk.

\begin{lemma}\label{lem: 0-1-liminf}
For  $x\in G$, it holds
$$\liminf_{n\to\infty}\max_{|v|=n}\frac{|X_v^x|}{n} = \liminf_{n\to\infty}\max_{|v|=n}\frac{|X_v|}{n}, \quad \text{almost surely}.$$
\end{lemma}

\begin{proof}
This is a direct consequence of the inequalities 
\[
|X_v|-|x|\leq |xX_v| \leq |x|+|X_v|, 
\]
where $v\in\GWT_n$ for some $n\in\N$.
\end{proof}

\subsection{$0$-$1$ law for the maximal displacement}\label{sec:0-1}

We establish a $0$-$1$ law for the maximal displacement, which enables us to lift the results of Subsection \ref{sec:lower} to an almost sure lower bound.

\begin{lemma}\label{lem: help 1}
Under assumptions \ref{A1} and \ref{A3}, for any $a\geq 0$ it holds
$$\Prob\Big(\liminf_{n\to\infty}\max_{|v|=n}\frac{|X_v|}{n}\geq a\Big)\in\lbrace 0,1\rbrace.$$
\end{lemma}
\begin{proof}
We consider the complementary event. Let $k\in\N$ be arbitrary, and for $u\in\GWT_k$ we write $\GWT_n^u$ for the set of descendants of $u$ at level $n$ in  $\GWT$. Since $\GWT_k$ is finite almost surely, we obtain
\begin{align*}
\liminf_{n\to\infty}\max_{|v|=n}\frac{|X_v|}{n}&=\max_{|u|=k}\Big(\liminf_{n\rightarrow\infty}\max_{v\in\GWT_n^u}\frac{|X_v|}{n}\Big)
=\max_{|u|=k}\Big(\liminf_{n\rightarrow\infty}\max_{v\in\GWT_n^u}\frac{|X_u^{-1}X_v|}{n-k}\Big),
\end{align*}
where above we have used Lemma \ref{lem: 0-1-liminf}. The sequence
$\Big(\liminf_{n\rightarrow\infty}\max_{v\in\GWT_n^u}\frac{|X_u^{-1}X_v|}{n-k}\Big)_{|u|=k}$
is a family of i.i.d.~random variables with the same distribution as
$\liminf_{n\rightarrow\infty}\max_{|v|=n}\frac{|X_v|}{n}$. If
$\mathcal{F}_k$ is the $\sigma$-algebra generated by the  first $k$ steps of the branching random walk, then
\begin{align*}
\mathbb{E}\bigg[\mathds{1}\Big\lbrace\liminf_{n\to\infty}\max_{|v|=n}\frac{|X_v|}{n}<a\Big\rbrace\Big|\mathcal{F}_k\bigg]&=\mathbb{E}\bigg[\mathds{1}\Big\lbrace\max_{|u|=k}\Big(\liminf_{n\rightarrow\infty}\max_{v\in\GWT_n^u}\frac{|X_u^{-1}X_v|}{n-k}\Big)<a\Big\rbrace\Big|\mathcal{F}_k\bigg]\\
&=\prod_{|u|=k}\Prob\Big(\liminf_{n\rightarrow\infty}\max_{v\in\GWT_n^u}\frac{|X_u^{-1}X_v|}{n-k}<a\Big)\\
&=\Prob\Big(\liminf_{n\to\infty}\max_{|v|=n}\frac{|X_v|}{n} < a\Big)^{|\GWT_k|}.
\end{align*}
In view of \ref{A1} and \ref{A3}, the branching process survives almost surely, so
$$\lim_{k\rightarrow\infty}\mathbb{E}\bigg[\Prob\Big(\liminf_{n\to\infty}\max_{|v|=n}\frac{|X_v|}{n} < a\Big)^{|\GWT_k|}\bigg]\in\{0,1\},$$
from which the claim follows.
\end{proof}
The Kolmogorov $0$-$1$ law is not applicable here: the "the tail" event inside the probability is typically not independent of the $\sigma$ algebra generated by $\mathcal F_v$. This is true if the underlying branching tree is deterministic.

\subsection{Upper bound}\label{sec:upper}

We derive an upper bound on the maximal displacement for the branching random walk on $G$. Using the many-to-one formula together with large deviation estimates for the underlying random walk $(Y_n)_{n\in\mathbb{N}}$ , we show that the probability a particle travels faster than $v_{\max}$ decays exponentially.
\begin{prop}\label{prop:upper}
Fix $a > v_{\max}$. Under Assumptions \ref{A1}–\ref{A3}, we have
$$\mathbb{P}\Big(\limsup_{n \to \infty} \max_{|v|=n} \frac{|X_v|}{n} \le a\Big) = 1.
$$
\end{prop}
\begin{proof}
Let $N_{a,n}$ denote the number of generation-$n$ particles that move at speed at least $a$, i.e.,
$$N_{a,n}:=\sum_{|v|=n}\mathds{1}\{|X_v|\geq na\}.$$
Lemma \ref{lem: many-to-one} gives 
$\E[N_{n,a}]=\rho^n\Prob(|Y_n|\geq na)$.
Since $a>v_{\max}$ it holds that $\inf_{x\geq a}I(x)>\log\rho$,
from which it follows 
$$\limsup_{n\rightarrow\infty}\frac{1}{n}\log\E[N_{n,a}]<0.$$
Hence there are constants $c, C > 0$ with $\mathbb{E}[N_{n,a}] \le C e^{-c n}$. By Markov’s inequality,
$$
\sum_{n=1}^\infty \mathbb{P}(N_{n,a} \ge 1)
\le \sum_{n=1}^\infty \mathbb{E}[N_{n,a}]
\le C \sum_{n=1}^\infty e^{-c n}
< \infty,
$$
and the Borel–Cantelli lemma yields the desired result.
\end{proof}

\subsection{Multitype branching process}\label{sec:multi-typ}

Recall that $v_{\max}$ is the largest intersection point of the rate function $I$ of the sequence $\big(\frac{|Y_n|}{n}\big)_{n\in\N}$ with the horizontal line $\log\rho$.
For this section, we fix some arbitrary $a\in[\ell,v_{\max})$. We show first a lower bound for the maximal displacement of the branching particles:
$$\Prob\bigg(\liminf_{n\to\infty}\max_{|v|=n}\frac{|X_v|}{n}\geq a\bigg)=1.$$
To this end, we construct a multitype branching process that consists of only those particles moving at speed greater than $a$, and then we show that this process survives with positive probability. Applying Lemma \ref{lem: help 1} then yields the required lower bound. The types are $\{1,2,\dots,r\}$, corresponding to the factors of the free product.
Fix a type $i$. Consider the branching random walk with offspring distribution $\pi$ started at the origin, and focus on particles that exit $\mathcal{B}_n$ quickly enough while staying within the cone $C(i)$. Among these fast particles, we classify types by the index of the group to which the last letter belongs. Let $Z_{ij}$ be the number of type-$j$ particles that leave $\mathcal{B}_n$ for the first time without exiting $C(i)$. This defines a multitype branching process with offspring matrix given by the random variables $(Z_{ij})_{i,j}$; see Figure \ref{fig:pdfimage} for a graphical representation.
By choosing $n \in \mathbb{N}$ sufficiently large, we ensure that this multitype process survives with positive probability, and the $0$-$1$ law completes the proof. We note that this approach is inspired by \cite{maxdisgalton}, where a branching process in a random environment is used instead of a multitype process, and we adopt a similar notation to \cite{maxdisgalton}. 
\begin{figure}
  \centering
  \includegraphics[width=0.55\linewidth]{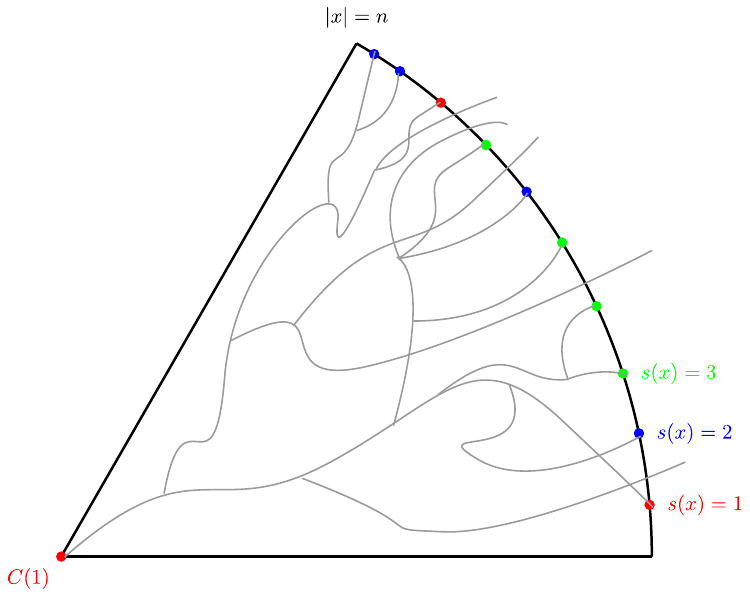}
  \caption{The multitype branching process on $\Z/2\Z*\Z/2\Z*\Z/2\Z$: red, blue, and green stand for $1$, $2$ and $3$, the first, second and third copy of $\Z/2\Z$, respectively. We consider all branching random walk particles started in $e$ and stopped after their first exit of $\mathcal{B}_n$. Then we keep only those particles that never left the cone $C(1)$, and reached $\mathcal{B}_n^c$ fast enough. We partition these particles according to their suffix: here we have $Z_{1,1} = 3$, $Z_{1,2} = 3$, and  $Z_{1,3} = 4$, and  $s(x)$ stands for the suffix of the marked displacement.}  
\label{fig:pdfimage}
\end{figure}

We now formalize the above heuristics.
Consider the branching random walk $(X_v)_{v \in \GWT}$ with offspring distribution $\pi$ and step distribution $\mu$. For $u \in \GWT$ and $n \in \mathbb{N}$, define the first time the ancestral line of $u$ exits $\mathcal{B}_n$ by
$$T_n^u := \inf\{1 \le k \le |u| : |X_{u_k}| \ge n\}.$$
We also define the stopping line of first exits from $\mathcal{B}_n$ as
$$\mathcal{L}(n) := \{u_{T_n^u} : u \in \GWT,\ T_n^u < \infty\},$$
where $u_{T_n^u}$ denotes the ancestor of $u$ in generation $T_n^u$.
Fix $i \in \{1,2,\dots,r\}$. Among the particles that hit $\mathcal{B}_{n}^c$ for the first time, we consider only those that exit quickly while staying inside the cone $C(i)$:
$$\mathcal{Z}_i(a,n):=
\Big\{u \in \mathcal{L}(n) :\ |u| \le \frac{n}{a},\ \forall\, k \in \{1,2,\dots,|u|\},\ X_{u_k} \in C(i)\Big\}.$$
We further distinguish the elements of $\mathcal{Z}_i(a,n)$ by type and, for each $j \in \{1,2,\dots,r\}$, set
$$\mathcal{Z}_{ij}(a,n) := \{u \in \mathcal{Z}_i(a,n) : s(X_u) = j\},$$
where $s(X_u)$ is the suffix of $X_u$, and denote the number of type-$j$ particles by
$$Z_{ij}(a,n) := \#\mathcal{Z}_{ij}(a,n).$$
We now consider an $r$-type branching process whose offspring distribution matrix is given by the random variables $Z_{ij}(a,n)$. Let $M(a,n) = \big(M_{ij}(a,n)\big)_{1 \le i,j \le r}$ denote its mean offspring matrix
$$ M_{ij}(a,n) := \mathbb{E}\big[ Z_{ij}(a,n) \big].$$
We denote by $(B^{m}(a,n))_{m \ge 0}$ the resulting multitype branching process with $r$ types and mean matrix $M(a,n)$. Our first goal is to show that, for $n \in \mathbb{N}$ sufficiently large, the spectral radius (largest eigenvalue) of $M(a,n)$ is strictly larger than $1$. This ensures that the process $(B^m(a,n))_{m \ge 0}$ survives with positive probability; see \cite{MR1991122} for more details. We will use the following simple fact.
\begin{lemma}\label{lem: helping perron-frob} If $A\in \R^{r\times r}$ is a positive matrix  with $\min_{1\leq i\leq r}\sum_{j=1}^r A_{ij} > 1$,
then the largest eigenvalue (Perron-Frobenius eigenvalue) of $A$ is bigger than $1$.
\end{lemma}
\begin{proof}
Due to the Perron-Frobenius theorem, the largest eigenvalue $\lambda$ of $A$ is positive and has a unique positive left eigenvector $v\in\R^{1\times r}$, such that 
$\Vert v \Vert_1 = \sum_{i=1}^r v_i = 1$. Then
\begin{align*}
\lambda = \Vert \lambda v\Vert_1 = \Vert vA \Vert_1 = \sum_{j=1}^r \sum_{i=1}^r v_iA_{ij} =\sum_{i=1}^r v_i \sum_{j=1}^r A_{ij} > \sum_{i=1}^r v_i = 1.
\end{align*}

\end{proof}
\begin{prop}\label{prop:perron}
Assuming \ref{A1}–\ref{A3}, there exists $n_0 \in \mathbb{N}$ such that for all $n \ge n_0$, the largest eigenvalue $\nu$ of $M(a,n)$ satisfies $\nu > 1$. Consequently, for each such $n$, the process $(B^m(a,n))_{m \ge 0}$ survives with positive probability.   
\end{prop}
\begin{proof}
Since
$\ell\leq a<v_{\max}\leq K := \sup\{|x|:\ x\in \textsf{supp}(\mu)\}$,
we can choose $n$ big enough such that every entry of $M(a,n)$ becomes strictly positive. For $\varepsilon>0$ we define a slightly modified version $\mathcal{Z}'_i(a,\varepsilon,n)$ of $\mathcal{Z}_i(a,n)$, as the set
$$
\Zcal'_i(a,\varepsilon,n):=\Big\lbrace u\in\Lcal(n): \ \frac{n}{a+\varepsilon}<|u| \leq\frac{n}{a},\ \forall k\in\{1,2,\dots,|u|\}:X_{u_k}\in C(i)\Big\rbrace.
$$ 
We now show that, for sufficiently small $\varepsilon > 0$, there exists $n_0 \in \mathbb{N}$ such that for all $n \ge n_0$,
$$
\min_{1 \le i \le r} \mathbb{E}\big[\#\mathcal{Z}'_i(a,\varepsilon,n)\big] > 1.
$$
For  fixed $i\in\{1,2,\dots,r\}$, we denote  by $E_{n,i}^v$  the event that 
$X_{v_k}\in C(i)$ for all $k\le T_n^v$. Together with Lemma \ref{lem: many-to-one}, we get
\begin{equation}\label{eq:many-one-aen}
\mathbb{E}\Big[\#\Big\{v\in \GWT_{\lceil n/a \rceil} : \frac{n}{a+\varepsilon}<T_n^v\leq \frac{n}{a},\ E_{n,i}^v\Big\}\Big]=\rho^{\lceil n/a \rceil}\Prob\Big(\frac{n}{a+\varepsilon}<T_n\leq \frac{n}{a}, \ E_{n,i}\Big).
\end{equation}
We bound the left-hand side from above by counting all descendants at level $\lceil n/a \rceil$ of those particles that first exited $\mathcal{B}_n$. We have
\begin{align*}\label{eq:prop-perron1}
\mathbb{E}\Big[\#\Big\{v\in &\GWT_{\lceil n/a \rceil} : \frac{n}{a+\varepsilon}<T_n^v\leq \frac{n}{a},\  E_{n,i}^v\Big\}\Big]\\ \notag
&\leq \E\Big[\sum_{u\in\Lcal(n)}\rho^{\lceil n/a \rceil-|u|}\mathds{1}\Big\{u\in\Lcal(n):\frac{n}{a+\varepsilon}<|u|\leq \frac{n}{a},\ E_{n,i}^v\Big\}\Big]\\
&\leq \rho^{\lceil n/a \rceil-\frac{n}{a+\varepsilon}}\mathbb{E}\Big[\#\Big\{u\in\Lcal(n):\frac{n}{a+\varepsilon}<|u|\leq \frac{n}{a}, \ E_{n,i}^v\Big\}\Big]\\
&=\rho^{\lceil n/a \rceil-\frac{n}{a+\varepsilon}}\E\big[\#\mathcal{Z}'_i(a,\varepsilon,n)\big].
\end{align*}
After taking the logarithm and dividing by $n$ in \eqref{eq:many-one-aen}, and using the upper bound from the previous inequality, we obtain
$$\frac{1}{n}\log\E[\#\mathcal{Z}'_i(a,\varepsilon,n)]\geq \frac{\log\rho}{a+\varepsilon}+\frac{1}{n}\log\Prob\Big(\frac{n}{a+\varepsilon}<T_n\leq \frac{n}{a},\ E_{n,i}\Big).$$
By taking the liminf on both sides above and using Lemma \ref{lem: LDP stoping time}, we finally get
\begin{align*}
\liminf_{n\rightarrow\infty}\frac{1}{n}\log\E[\#\mathcal{Z}'_i(a,\varepsilon,n)]&\geq\frac{\log\rho}{a+\varepsilon}+\liminf_{n\rightarrow\infty}\frac{1}{n}\log\Prob\Big(\frac{n}{a+\varepsilon}<T_n\leq \frac{n}{a},\ E_{n,i}\Big)\\
&= \frac{\log\rho}{a+\varepsilon}-\frac{I(a)}{a}=\frac{1}{a}(\log\rho-I(a))-\frac{\varepsilon}{a(a+\varepsilon)}\log\rho.
\end{align*}
Since $a<v_{\max}$, we have $\log\rho-I(a)>0$, thus by choosing $\varepsilon>0$ small enough it holds
$\liminf_{n\rightarrow\infty}\frac{1}{n}\log\E[\#\mathcal{Z}'_i(a,\varepsilon,n)]>0.$
This implies that $\E[\#\mathcal{Z}'_i(a,\varepsilon,n)]$ grows at least exponentially for $\varepsilon>0$ chosen suitably small, and the growth rate can be chosen independently of $i$. So there exists $n_0\in\N$, such that
$\E[\#\mathcal{Z}'_i(a,\varepsilon,n)]>1$,
for all $n\geq n_0$ and $i\in\{1,2,\dots,r\}$. Since $\Zcal'_i(a,\varepsilon,n)\subset\Zcal_i(a,n)$, for all $n\geq n_0$ it holds
$$
\sum_{j=1}^r M_{ij}(a,n) = \E[\#\Zcal_i(a,n)]\geq \E[\#\Zcal'_i(a,\varepsilon,n)]>1,
$$ 
which together with
Lemma \ref{lem: helping perron-frob} completes the proof.
\end{proof}

\subsection{Lower bound}\label{sec:lower}
In this section we prove the lower bound for the maximal displacement of branching particles.
\begin{prop}\label{prop:lower}
Fix $a\in [\ell,v_{\max}).$ Under Assumptions \ref{A1}-\ref{A3}, we have
$$\Prob\Big(\liminf_{n\rightarrow\infty}\max_{|v|=n}\frac{|X_v|}{n}\geq a\Big)=1.$$
\end{prop}
\begin{figure}[t]
\centering
\scalebox{.65}{\begin{tikzpicture}[-,>=stealth,auto,node distance=3cm,
  thick,main node/.style={circle,draw,font=\sffamily\Large\bfseries}]

\draw (-5.25,-5.25) rectangle (-2.25,-2.25);
\draw[red] (0,0) -- (-2.25,-2.25);
\draw[red] (-5.25,-2.25) -- (-7.5,0);
\draw[red] (-5.25,-5.25) -- (-7.5,-7.5);
\draw[red] (-2.25,-5.25) -- (0,-7.5);

\draw (0,0) rectangle (3,3);

\draw (3.75,-0.75) rectangle (5.25,-2.25);
\draw[red] (5.25,5.25) -- (6,6);
\draw[red] (3.75,5.25) -- (3,6);
\draw[red] (3.75,3.75) -- (3,3);
\draw[red] (5.25,3.75) -- (6,3);
\draw (3.75,3.75) rectangle (5.25,5.25);
\draw[red] (5.25,-2.25) -- (6,-3);
\draw[red] (3.75,-2.25) -- (3,-3);
\draw[red] (3.75,-0.75) -- (3,0);
\draw[red] (5.25,-0.75) -- (6,0);
\draw (-0.75,3.75) rectangle (-2.25,5.25);
\draw[red] (-0.75,5.25) -- (0,6);
\draw[red] (-2.25,5.25) -- (-3,6);
\draw[red] (-2.25,3.75) -- (-3,3);
\draw[red] (-0.75,3.75) -- (0,3);

\draw [dashed, gray]
(-6.9,5.1) .. controls (0.3,0.3) .. (5.1,-6.9);
\node[red] at (-5,5) {\huge $C(1)$};
\node[black] at (-6,3) {\huge $C(2)$};
\node[red] at (-2.5,-1) {\large $G_1 = \Z/2\Z$};
\node[black] at (1.5,1.5) {\large $G_2 = \Z/4\Z$};
\node[gray] at (0.25,0.25) {$e$};
\node[gray] at (0.15,3.25) {$b$};
\node[gray] at (2.85,3.25) {$b^2$};
\node[gray] at (3.25,0.25) {$b^3$};

\node[gray] at (-0.5,4) {$ba$};
\node[gray] at (-2.6,4) {$bab$};
\node[gray] at (-2.6,5) {$bab^2$};
\node[gray] at (-0.4,5) {$bab^3$};
\node[gray] at (3.3,4.1) {$b^2a$};
\node[gray] at (3.3,5) {$b^2ab$};
\node[gray] at (5.7,5) {$b^2ab^2$};
\node[gray] at (5.7,4.1) {$b^2ab^3$};
\node[gray] at (3.3,-1){$b^3a$};
\node[gray] at (5.7,-1){$b^3ab$};
\node[gray] at (3.3,-1.9){$b^3ab^3$};
\node[gray] at (5.7,-1.9){$b^3ab^2$};

\node[gray] at (-2,-2.6){$a$};
\node[gray] at (-4.9,-2.5){$ab^3$};
\node[gray] at (-2,-5){$ab$};
\node[gray] at (-4.9,-4.9){$ab^2$};

\node[gray] at (-6.7,0){$ab^3a$};
\node[gray] at (-6.3,-7){$ab^2a$};
\node[gray] at (0,-7){$aba$};
\end{tikzpicture}}
\caption{The two cones $C(1)$ and $C(2)$ in the group $G=\Z/ 2\Z* \Z/ 4\Z$, with $\Z/2\Z = \{e,a\}$ and $\Z/4\Z=\{e,b,b^2,b^3\}$. At each element of $\Z/ 4\Z$ there is a copy of $C(2)$ attached, and at each element of $\Z/2\Z$ there is a copy of $C(1)$ attached.}\label{fig:free-prod-cones}
\end{figure}
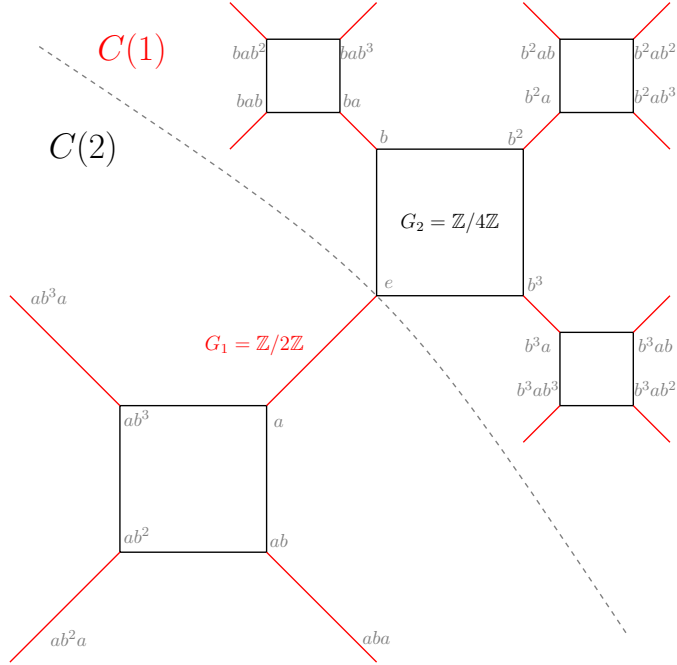
To prove Proposition \ref{prop:lower}, we relate the multitype branching process from Section \ref{sec:multi-typ} to the branching random walk $(X_v)_{v\in\GWT}$ on $G$ as follows. For $u\in\GWT$ and $v\geq u$ we define
$$T_n^{u,v}:=\inf\bigl\{k\in \{|u|,|u|+1,\dots,|v|\}:\ d(X_{v_k},X_u)\geq n\bigr\}-|u|,$$
which represents the first time at which a particle that is both a descendant of $u$ and an ancestor of $v$ exits $\mathcal{B}_n$. We subtract $|u|$ in order to reset the clock for the descendants of $u$. We also define the set
$$\Lcal_u(n) = \big\{v_{T_{n}^{u,v}+|u|}:\ v\geq u,\ T_{n}^{u,v}<\infty \big\}$$
of  descendants of $u$ that are frozen when they first exit the ball $\mathcal{B}_n$.
For every $j\in\{1,2,\dots,r\}$ and $u\in\GWT$ we set
$$
\Zcal_u(a,n):=\Big\{v\in\Lcal_u(n):\ |v|-|u|\leq\frac{n}{a},\ \forall k\in\{|u|,|u|+1,\dots,|v|\}:X_u^{-1}X_{v_k}\in C(s(X_u)) \Big\}$$
for the descendants $v$ of $u$ in $\mathcal{L}(u)$ that have not moved too far from $u$ and have remained within the cone $C(i)$, where $i$ denotes the suffix of $X_u$. We then partition this set according to the suffixes of $X_v$, for $v \in \mathcal{Z}_u(a,n)$ as 
$$\Zcal_{u,j}(a,n):=\{v\in\Zcal_u(a,n):\ s(X_v)=j\}.$$
For each $i \in \{1,2,\dots,r\}$ and $a \in [\ell, v_{\max})$, define
$$
\mathcal{N}_i^0(a,n) := \{ v \in \GWT_1 : s(X_v) = i \}.$$
We define the sets $\mathcal{N}_i^m(a,n)$  for $m\in\N$ and $i\in\{1,2,\dots,r\}$ inductively as follows. Given $\mathcal{N}_i^{m-1}(a,n)$ for every $i\in\{1,2,\dots,r\}$, we set
\begin{align*}
\mathcal{N}_i^m(a,n):=\bigcup_{j=1}^r\bigcup_{u\in\mathcal{N}_j^{m-1}(a,n)} \mathcal{Z}_{u,i}(a,n),
\end{align*}
and for the cardinality of $\mathcal{N}_i^m(a,n)$ we write $N_i^m(a,n)$. Finally,  write 
$$\mathcal{N}^m(a,n):= \bigcup_{i=1}^r \mathcal{N}_i^m(a,n)\quad
\text{and} \quad 
N^m(a,n):= \sum_{i=1}^r N_i^m(a,n).$$
The next result establishes that the sequence $\big((N_i^m(a,n))_{1\leq i\leq r} \big)_{m\geq 1}$ is a version of the multitype branching process constructed in Section \ref{sec:multi-typ}, with a suitably chosen initial state.

\begin{lemma}\label{lem: Z is N}
Assuming \ref{A3}, for any $a\in[\ell,v_{\max})$ and $n\in\N$, the sequence $\big((N_i^m(a,n))_{1\leq i\leq r} \big)_{m\geq 1}$ is a multitype branching process with initial state $
(N_i^0(a,n))_{1\leq i\leq r}$
and mean offspring matrix $M(a,n)$.
\end{lemma}
\begin{proof}
Let $\GWT^u$ be the subtree of $\GWT$ rooted at $u$. Denote by $|\cdot|_u$ the distance to this root within $\GWT^u$, i.e., $|\cdot|_u = |\cdot| - |u|$. For $0 \le k \le |v|_u$, define $v_k^u := v_{|u|+k}$, the ancestor of $v$ at generation $k$ in $\GWT^u$. For $v=uw$ we have
	\begin{align*}
		T_n^{v,u} &= \inf\{|u|\le k\le |v|:\ d(X_{v_k},X_u)\ge n\}-|u|\\
		&= \inf\{0\le k\le |v|-|u|:\ |X_u^{-1}X_{v_{|u|+k}}|\ge n\}\\
		&= \inf\{0\le k\le |w|:\ |X^u_{uw_k}|\ge n\}.
	\end{align*}
Thus $T_n^{v,u}$ depends only on $(X_u^{-1}X_{u'})_{u'\in\GWT^u}$ and is defined from these variables in the same way as $T_n^w$ is defined from $(X_{u'})_{u'\in\GWT}$. Similarly, if $s(X_u)=i$, then the set $\{w: uw\in\mathcal{Z}_{u,j}(a,n)\}$ depends only on $(X_u^{-1}X_{u'})_{u'\in\GWT^u}$ and is defined in the same way as $\mathcal{Z}_{ij}(a,n)$ from $(X_{u'})_{u'\in\GWT}$.

For any $m$, the set $\mathcal{N}_i^m(a,n)$ is an optional stopping line in the sense of Jagers \cite{Jagers}. The notation $v\le \mathcal{N}_i^m(a,n)$ means that $v$ lies on the geodesic between the root $\emptyset$ and some element of $\mathcal{N}_i^m(a,n)$.  Conditioned on $(X_v)_{v\le \mathcal{N}_i^m(a,n)}$, the processes
	\[
	(X_v^u)_{v\in\GWT^u} := (X_u^{-1}X_v)_{v\in\GWT^u}, \qquad u\in \mathcal{N}_i^m(a,n),
	\]
	are, by the strong Markov branching  property \cite[Theorem 4.14]{Jagers}, independent and have the same distribution as $(X_v)_{v\in\GWT}$. Consequently, conditioned on $(X_v)_{v\le \mathcal{N}_i^m(a,n)}$, the evolutions of $\mathcal{N}_i^{m'}(a,n)\cap\GWT^u$ for different $u\in \mathcal{N}_i^m(a,n)$ are independent and their law depends only on $s(X_u)$. Hence $\big((N_i^m(a,n))_{1\le i\le r}\big)_{m\ge1}$ is a multitype Galton--Watson process.
\end{proof}


Consequently, for $n$ sufficiently large, the sequence $\mathcal{N}^m(a,n)$ remains nonempty with positive probability. Furthermore, it consists only of particles lying outside $\mathcal{B}_{n m}$ that have moved at a prescribed speed, as established in the next result.

\begin{lemma}\label{lem : N dist and speed}
For every $m\geq 0$ and $v\in\mathcal{N}^m(a,n)$ it holds that
$|v| \leq \frac{nm}{a}+1$ and $nm\leq |X_v|$.
\end{lemma}
\begin{proof}
Follows directly from the definition of the sets $\mathcal{N}^m(a,n)$.
\end{proof}

We are now in position to prove Proposition \ref{prop:lower}. The multitype branching process is defined by considering only those particles for which the stopping time of moving far is sufficiently small. We show below that a small stopping time is equivalent to the maximal displacement of the branching random walk being large.

\begin{proof}[Proof of Proposition \ref{prop:lower}]
Fix $a \in [\ell, v_{\max})$. Choose $n$ sufficiently large so that $\mathcal{N}^m(a,n)$ stays nonempty with positive probability. On this event, there exists a sequence $(v_m)_{m \ge 1}$ with $v_m \in \mathcal{N}^m(a,n)$ and $v_{m-1}$ an ancestor of $v_m$.
For each $l$, let $u_l \in \GWT$ be the unique generation-$l$ vertex lying on the ray $(v_m)_{m \ge 1}$. If $u_l$ lies on the segment $[v_m,v_{m+1}]$, then by Lemma~\ref{lem : N dist and speed}
	\begin{align*}
		l \le |v_{m+1}|\le \frac{n(m+1)}{a}+1
		\qquad\text{and}\qquad
		|X_{u_l}|\ge |X_{v_m}|-nK\ge mn-nK.
	\end{align*}
	Combining these inequalities we obtain
$|X_{u_l}|\ge al-(K+1)n-a$, and thus
	\begin{align*}
		\liminf_{l\to\infty}\max_{|u|=l}\frac{|X_u|}{l}
		\ge 
		\liminf_{l\to\infty}\frac{|X_{u_l}|}{l}
		\ge a.
\end{align*}
We have thus shown
$$\Prob\Big(\liminf_{l\rightarrow\infty}\max_{|u|=l}\frac{|X_u|}{l}\geq a\Big)>0,$$
and the claim follows from the $0$-$1$ law from Lemma \ref{lem: help 1}.
\end{proof}

\begin{proof}[Proof of Theorem \ref{thm:main}]
From Proposition \ref{prop:upper}, because $a>v_{\max}$ was arbitrary, we get
\begin{equation}\label{eq:max-upbnd}
\Prob\Big(\limsup_{n\rightarrow\infty}\max_{|v|=n}\frac{|X_v|}{n}\leq v_{\max}\Big)=1.
\end{equation}
Similarly, from Proposition \ref{prop:lower}, since $a\in [l,v_{\max})$ was arbitrary, we also get 
\begin{equation}\label{eq:max-lobnd}
\Prob\Big(\liminf_{n\rightarrow\infty}\max_{|v|=n}\frac{|X_v|}{n}\geq v_{\max}\Big)=1.
\end{equation}
Equations \eqref{eq:max-upbnd} and \eqref{eq:max-lobnd} together prove the result.
\end{proof}

\section*{Minimal displacement}\label{sec: min dis}

In this section we prove Theorem \ref{thm:main-min}. The lower bound is straightforward and similar to the argument used for the upper bound in Theorem \ref{thm:main}. Recall the definition of $v_{\min} \ge 0$ from Equation \eqref{eq:v-min}.
\begin{prop}
Assume $v_{\min} > 0$ and fix $a \in [0, v_{\min})$. Under Assumptions \ref{A1}–\ref{A3}, we have
$$
\mathbb{P}\Big(\liminf_{n \to \infty} \min_{|v|=n} \frac{|X_v|}{n} \ge a\Big) = 1.
$$
\end{prop}
\begin{proof}
The proof is similar to the one of Proposition \ref{prop:upper}. For $n\in\N$, define the random variable
$$\widehat{N}_{n,a} := \sum_{v\in\GWT_n}\mathds1 \{|X_v|\leq na\},$$
which counts the number of particles in generation $n$ with displacement at most $na$. Lemma \ref{lem: many-to-one} gives
$\E[\widehat{N}_{n,a}] = \rho^n\Prob(|Y_n|\leq na)$.
Since $a<v_{\min}$ it holds that $\inf_{x\leq a}I(a)>\log\rho$, from which it follows 
$$\limsup_{n\to\infty}\frac{1}{n}\log\E[\widehat{N}_{n,a}]<0.$$
Consequently, there exist constants $c, C > 0$ with $\mathbb{E}[\widehat{N}_{n,a}] \le C e^{-c n}$. By Markov’s inequality,
$$
\sum_{n=1}^\infty \mathbb{P}(\widehat{N}_{n,a} \ge 1)
\le \sum_{n=1}^\infty \mathbb{E}[\widehat{N}_{n,a}]
\le C \sum_{n=1}^\infty e^{-c n}
< \infty.
$$
Applying the Borel–Cantelli lemma finishes the proof.
\end{proof}
The upper bound is considerably simpler than the lower bound for the maximal diplacement in Proposition \ref{prop:lower}. Indeed, for any $a > v_{\min}$, if $v \in \mathbb{T}_n$ and $w \in \mathbb{T}_{n+m}$ with $v < w$, and if $|X_v| \le n a$ and $d(X_v, X_w) \le m a$, then by the triangle inequality,
$$|X_w| \le |X_v| + d(X_v, X_w) \le n a + m a \le (n + m) a.$$
In contrast, the conditions $|X_v| \ge n a$ and $d(X_v, X_w) \ge m a$ do not imply $|X_w| \ge (n + m) a$.
\begin{prop}
Fix $a > v_{\min}$. Under Assumptions \ref{A1}–\ref{A3}, we have
$$
\mathbb{P}\Big(\limsup_{n \to \infty} \min_{|v|=n} \frac{|X_v|}{n} \le a\Big) = 1.
$$
\end{prop}
\begin{proof}
Fix $a > v_{\min}$ and, for $n \in \mathbb{N}$, set
$$
\widehat{\mathcal{Z}}_1(a,n) := \{ v \in \mathbb{T}_n : |X_v| \le n a \},
$$
and for $u \in \mathbb{T}$ define
$$
\widehat{\mathcal{Z}}_u(a,n) := \{ v \in \mathbb{T}_{|u| + n} : v > u,\ d(X_v, X_u) \le n a \}.
$$
Define the families $\widehat{\mathcal{Z}}_m(a,n)$ inductively: given $\widehat{\mathcal{Z}}_m(a,n)$ for $m \in \mathbb{N}$, set
$$
\widehat{\mathcal{Z}}_{m+1}(a,n) = \bigcup_{u \in \widehat{\mathcal{Z}}_m(a,n)} \widehat{\mathcal{Z}}_u(a,n).
$$
Then the sequence of cardinalities $(\widehat{Z}_m(a,n))_{m \ge 0} := (\#\widehat{\mathcal{Z}}_m(a,n))_{m \ge 0}$ forms a Galton–Watson process whose offspring distribution is the law of $Z_1(a,n)$. By Lemma \ref{lem: many-to-one},
$$
\mathbb{E}[\widehat{Z}_1(a,n)] = \rho^n \,\mathbb{P}(|Y_n| \le n).
$$
Since $a>v_{\min}$, it holds $\inf_{x<a}I(x)>\log\rho$, and so 
$$\liminf_{n\to\infty}\frac{1}{n}\log\E[\widehat{Z}_1(a,n)]>0.$$
Therefore, for sufficiently large $n \in \mathbb{N}$ we have $\mathbb{E}[Z_1(a,n)] > 1$, so $(\widehat{Z}_m(a,n))_{m \ge 0}$ is supercritical. Hence, for such $n$, the event
$\widehat{\mathcal{Z}}_m(a,n) \neq \varnothing$ for all $m \in \mathbb{N}$
has positive probability. Condition on this event and pick $u \in \widehat{\mathcal{Z}}_m(a,n)$. By the definition of the branching process $(\widehat{Z}_m(a,n))_{m \ge 0}$, we have $|u| = n m$, and for every $0 \le j < m$, we have $d(X_{u_{n j}}, X_{u_{n(j+1)}}) \le n a$.
Applying the triangle inequality,
$$
|X_u|
\le |X_{u_n}| + d(X_{u_n}, X_{u_{2n}}) + \cdots + d(X_{u_{(m-1)n}}, X_{u_{mn}})
\le m n a.
$$
Consequently,
$\min_{|v| = n m} |X_v| \le |X_u| \le n m a $.
Since the walk can jump only a bounded distance $K$ at each step we obtain that
$\limsup_{n\to\infty}\min_{|v|=n}\frac{|X_v|}{n}\leq a$
with positive probability. The $0$-$1$ law for the liminf of the maximal displacement (Theorem \ref{lem: 0-1-liminf}) can be adapted directly to establish a similar $0$-$1$ law for the limsup of the minimal displacement. This completes the proof.
\end{proof}

\textbf{Question.} Our main results, Theorems \ref{thm:main} and \ref{thm:main-min}, identify the leading linear behavior of the maximal and minimal displacements for branching random walks on free products of finite groups. A natural next question concerns the second-order term: what can be said about the fluctuations of $\max_{|u|=n} |X_u| - n v_{\max}$ and $\min_{|u|=n} |X_u| - n v_{\min}$?

\textbf{Acknowledgements.} We are very grateful to the two anonymous referees for a careful reading of a previous version of the paper, and for the many valuable comments and suggestions, which substantially improved the paper.

\textbf{Funding information.} The research of E. Sava-Huss was funded in part by the Austrian Science Fund (FWF) [10.55776/PAT3123425]. For open access purposes, the authors have applied a CC BY public copyright license to any author-accepted manuscript version arising from this submission.

\bibliographystyle{amsalpha}
\bibliography{lit}

\providecommand{\bysame}{\leavevmode\hbox to3em{\hrulefill}\thinspace}
\providecommand{\MR}{\relax\ifhmode\unskip\space\fi MR }
\providecommand{\MRhref}[2]{%
  \href{http://www.ams.org/mathscinet-getitem?mr=#1}{#2}
}
\providecommand{\href}[2]{#2}
\begin{thebibliography}{KKSH26}

\bibitem[A\"13]{maxconv}
E.~A\"id\'ekon, \emph{Convergence in law of the minimum of a branching random walk}, Ann. Probab. \textbf{41} (2013), no.~3A, 1362--1426.

\bibitem[ABR09]{logcorr2}
L.~Addario-Berry and B.~Reed, \emph{Minima in branching random walks}, Ann. Probab. \textbf{37} (2009), no.~3, 1044--1079.

\bibitem[BCZ24]{boundaryconv}
D.~Bertacchi, E.~Candellero, and F.~Zucca, \emph{Martin boundaries and asymptotic behavior of branching random walks}, Electron. J. Probab. \textbf{29} (2024), Paper No. 138, 28.

\bibitem[BGGS25]{maxdisgalton}
J.~Berestycki, N.~Gantert, D.~Geldbach, and Q.~Shi, \emph{Biased branching random walks on {B}ienaym\'e--{G}alton--{W}atson trees}, 2025.

\bibitem[Big77]{addmart2}
J.~D. Biggins, \emph{Martingale convergence in the branching random walk}, J. Appl. Probability \textbf{14} (1977), no.~1, 25--37.

\bibitem[BK04]{derivmart}
J.~D. Biggins and A.~E. Kyprianou, \emph{Measure change in multitype branching}, Adv. in Appl. Probab. \textbf{36} (2004), no.~2, 544--581.

\bibitem[BP94]{tree-index}
I.~Benjamini and Y.~Peres, \emph{Tree-indexed random walks on groups and first passage percolation}, Probab. Theory Related Fields \textbf{98} (1994), no.~1, 91--112.

\bibitem[CGM12]{BRWOnFreeProducts}
E.~Candellero, L.~Gilch, and S.~M\"uller, \emph{Branching random walks on free products of groups}, Proc. Lond. Math. Soc. (3) \textbf{104} (2012), no.~6, 1085--1120.

\bibitem[Che15]{maxbrown}
X.~Chen, \emph{Scaling limit of the path leading to the leftmost particle in a branching random walk}, Theory Probab. Appl. \textbf{59} (2015), no.~4, 567--589.

\bibitem[Cor21]{largedev}
E.~Corso, \emph{Large deviations for random walks on free products of finitely generated groups}, Electron. J. Probab. \textbf{26} (2021), Paper No. 134, 22.

\bibitem[CS86]{RWOnFreeProducts}
D.~Cartwright and P.~M. Soardi, \emph{Random walks on free products, quotients and amalgams}, Nagoya Math. J. \textbf{102} (1986), 163--180.

\bibitem[DZ98]{largedev_dembo}
A.~Dembo and O.~Zeitouni, \emph{Large deviations techniques and applications}, second ed., Applications of Mathematics (New York), vol.~38, Springer-Verlag, New York, 1998.

\bibitem[Gil07]{RateOfEscapeFreeProduct}
L.~Gilch, \emph{Rate of escape of random walks on free products}, J. Aust. Math. Soc. \textbf{83} (2007), no.~1, 31--54.

\bibitem[Gil11]{AssymptoticEntropyFreeProduct}
\bysame, \emph{Asymptotic entropy of random walks on free products}, Electron. J. Probab. \textbf{16} (2011), no. 3, 76--105.

\bibitem[GW86]{LLNRWFreeProduct}
P.~Gerl and W.~Woess, \emph{Local limits and harmonic functions for nonisotropic random walks on free groups}, Probab. Theory Relat. Fields \textbf{71} (1986), no.~3, 341--355.

\bibitem[Ham74]{maxdisreal1}
J.~M. Hammersley, \emph{Postulates for subadditive processes}, Ann. Probability \textbf{2} (1974), 652--680.

\bibitem[Har02]{MR1991122}
T.~E. Harris, \emph{The theory of branching processes}, Dover Phoenix Editions, Dover Publications, Inc., Mineola, NY, 2002.

\bibitem[HR17]{manytofew}
S.~Harris and M.~Roberts, \emph{The many-to-few lemma and multiple spines}, Ann. Inst. Henri Poincar\'e{} Probab. Stat. \textbf{53} (2017), no.~1, 226--242.

\bibitem[HS09]{logcorr1}
Y.~Hu and Z.~Shi, \emph{Minimal position and critical martingale convergence in branching random walks, and directed polymers on disordered trees}, Ann. Probab. \textbf{37} (2009), no.~2, 742--789.

\bibitem[Jag89]{Jagers}
P.~Jagers, \emph{General branching processes as markov fields}, Stoch. Process. Appl. \textbf{32} (1989), no.~2, 183--212.

\bibitem[Kin75]{maxdisreal2}
J.~F.~C. Kingman, \emph{The first birth problem for an age-dependent branching process}, Ann. Probability \textbf{3} (1975), no.~5, 790--801.

\bibitem[KKSH26]{limit2}
R.~Kaiser, M.~Klötzer, and E.~Sava-Huss, \emph{Limit theorems for the empirical distribution of supercritical branching random walks on transitive graphs}, Electron. J. Probab. \textbf{31} (2026), 1--29.

\bibitem[KW23]{limit1}
V.~Kaimanovich and W.~Woess, \emph{Limit distributions of branching {M}arkov chains}, Ann. Inst. Henri Poincar\'e{} Probab. Stat. \textbf{59} (2023), no.~4, 1951--1983.

\bibitem[Lal93]{FiniteRangeFreeProduct}
S.~Lalley, \emph{Finite range random walk on free groups and homogeneous trees}, Ann. Probab. \textbf{21} (1993), no.~4, 2087--2130.

\bibitem[LMW24]{Multifractal}
S.~Lai, H.~Ma, and L.~Wang, \emph{Multifractal spectrum of branching random walks on free groups}, 2024, 10.48550/arXiv.2409.01346.

\bibitem[Nev88]{addmart1}
J.~Neveu, \emph{Multiplicative martingales for spatial branching processes}, Progr. Probab. Statist., vol.~15, Birkh\"auser Boston, Boston, MA, 1988, pp.~223--242.

\bibitem[Shi15]{BRW}
Z.~Shi, \emph{Branching random walks}, Lecture Notes in Mathematics, vol. 2151, Springer, Cham, 2015.

\bibitem[Sta66]{stam}
A.~J. Stam, \emph{On a conjecture by {H}arris}, Z. Wahrscheinlichkeitstheorie und Verw. Gebiete \textbf{5} (1966), 202--206.

\bibitem[SWX23]{MR4630601}
V.~Sidoravicius, L.~Wang, and K.~Xiang, \emph{Limit set of branching random walks on hyperbolic groups}, Comm. Pure Appl. Math. \textbf{76} (2023), no.~10, 2765--2803.

\end{thebibliography}

\texttt{Robin Kaiser}, Technische Universität München, Germany. \texttt{ro.kaiser@tum.de}\\
\texttt{Martin Klötzer}, Universität Innsbruck, Innsbruck, Austria. \texttt{Martin.Kloetzer@uibk.ac.at}\\
\texttt{Konrad Kolesko}, Wroclaw University of Science and Technology, Poland. \texttt{konrad.kolesko@pwr.edu.pl}\\
\texttt{Ecaterina Sava-Huss}, Universität Innsbruck, 
Innsbruck, Austria. \texttt{Ecaterina.Sava-Huss@uibk.ac.at}

\end{document}